\documentclass[11pt]{amsproc}
\usepackage[margin=3cm]{geometry}
\usepackage{amssymb, latexsym, amsmath,enumerate,color}
\usepackage{euscript}

\def\theenumi{\roman{enumi}}
\def\p@enumi{\theenumi\theenumi}

\newtheorem{prop}{Proposition}
\newtheorem{cor}{Corollary}
\newtheorem{thm}{Theorem}

\theoremstyle{definition}

\def\QED{\hfill\vrule height .9ex width .8ex depth -.1ex \bigskip}

\newcommand{\new}[1]{{\emph {#1}}}
\newcommand{\aut}{{\sf Aut}}
\newcommand{\sym}{{\sf Sym}}

\newcommand{\F}{\mathbb{F}}

\newcommand{\Z}{\mathbb{Z}}

\newcommand{\New}[1]{{\it #1}}

\newcommand{\bC}{{\bf C}}

\newcommand{\sg}[1]{\langle {#1}\rangle}

\newcommand{\normaleq}{\trianglelefteq}

\newcommand{\comment}[1]{}
\newcommand{\cC}{{\mathcal C}}
\def\f{\EuScript}
\newcommand{\cD}{{\mathcal D}}
\newcommand{\paut}{{\sf PAut}}

\newcommand{\cay}{{\sf Cay}}

\begin{document}

\title{A solution of an equivalence problem for semisimple cyclic codes}
\author{Mikhail Muzychuk}
\address{Netanya Academic College, Netanya, Israel}
\email{muzy@netanya.ac.il}

\date{}

\begin{abstract}
In this paper we propose an efficient solution of an equivalence problem for semisimple
cyclic codes.
\end{abstract}

\maketitle

\section{Introduction}

Recall that an $[n,k]_q$ code is a $k$-dimensional subspace $\cC$ of $\F_q^n$.
Two $[n,k]_q$ codes $\cC,\cD\leq\F_q^n$ are called \New{(permutation) equivalent }, notation $\cC\sim\cD$, if one of them may be obtained from another one by permutting the coordinates. A linear code is called \New{cyclic} if it is invariant under a cyclic shift of the coordinates. An equivalence problem for cyclic codes, and, more generally, an isomorphism problem for arbitrary cyclic objects, was studied by many authors during last three decades 
\cite{babai,palfy,klp1,phelps,brand,HJP,evdopon,li2,dobson,mu99} - to mention a few. 
In the paper \cite{HJP} Huffman, Job and Pless  completely solved an isomorohism problem for arbitrary cyclic combinatorial objects on $p^2$ points where $p$ is a prime. The solution was  given in terms of \New{generalized multipliers} and was generalized in \cite{mu04} via a notion of a \New{solving set}: a set of permutations $S$ is called a solving set for a class of cyclic objects when two cyclic objects from the class are equivalent  if and only if they are equivalent by a permutation from $S$. It was shown in \cite{mu04} that there exists a solving set for colored circulant digraphs of order $n$ of cardinality $O(n^2\varphi(n))$. Moreover this set may be efficiently constructed from $n$ without any additional information.  The main result of this paper states that a solving set
constructed in \cite{mu04} is also good for cyclic semisimple codes. To formulate precise results we need more definitions.

Each cyclic code of length $n$ over finite field $\F_q$ may be considered as an ideal in the group algebra $\F_q[H]$ of a cyclic group $H$ of order $n$.  Notice that if $H$ is an arbitrary group, then a \New{group code} is an arbitrary right ideal of $\F_q[H]$.
In what follows we write $I\normaleq \F_q[H]$ to designate the fact that $I$ is a right ideal of $\F_q[H]$. A group code $I\normaleq \F_q[H]$ is called \New{semisimple} if the group algebra $\F_q[H]$ is semisimple, that is $\gcd(q,|H|)=1$. An automorphism group of a group code always contains a subgroup $H_R$ consisting of right tranlsations by the elements of $H$. This group acts regularly on the coordinates of a group code. Thus a group code is a particular case of a \New{Cayley combinatorial object} introduced by L.Babai \cite{babai}. Recall that a Cayley combinatorial object over a group $H$ is any relational structure on $H$ invariant under the group $H_R$.
Let $\f K$ be a class of Cayley combinatorial objects over the group $H$. Two objects from ${\f K}$ are called \New{Cayley isomorphic} if there exists an automorphism $H$ which maps one of them onto another. An object $K\in{\f  K}$ is called a \New{CI-object} if any $K'\in{\f K}$ isomorphic to $K$ is also Cayley isomorphic to $K$.  A group $H$ is called a \New{CI-group} with respect to the class $\f K$ if any object $K\in{\f K}$ is a CI-object. Two classes of Cayley objects are essential for this paper: group codes and \New{colored Cayley digraphs}. 

A colored Cayley digraph over a finite group $H$ is a pair $(H,\phi)$ where $\phi:H\rightarrow C$ is a function to the set of colors $C$. An arc $(x,y)\in H\times H$ is colored by a color $\phi_{xy^{-1}}$. We denote the corresponding colored Cayley digraph as $\cay(H,\phi)$. An isomorphism between two colored Cayley digraphs is defined in a natural way (see the next section where all related definitions are given). In the case when a coloring set $C$ is a finite field we identify the coloring function $\phi:H\rightarrow \F_q$ with an element of a group algebra $\sum_{h\in H}\phi(h)h$.

Now we are able to formulate the  main result of the paper.
\begin{thm}\label{main2} Let $H$ be a cyclic group of order $n$ written multiplicatively. Let $I\normaleq\F_q[H]$ be a semisimple cyclic code over $H$ generated by the idempotent $e = \sum_{h\in H} e_h h\in\F_q[H]$.  Then a solving set for $\cay(H,e)$ is a solving set for $I$.
\end{thm}

It was shown in \cite{mu04} that a solving set $S_e$ for a given colored Cayley digraph $\cay(\Z_n,e)$ over a cyclic group of order $n$ contains at most $\varphi(n)$ permutations. This set depends only on a partition ${\f P}_e$ of $H$ constructed from the coloring $e$ in the following way: two elements $a,b\in H$ belong to the same class of ${\f P}_e$ whenever $e_a = e_b$. Once a partition ${\f P}_e$ is built, a construction of the related solving set requires $O(n^2)$ arithmetic operations in the ring $\Z_n$. When the set $S_e$ is produced an equivalence testing becomes rather simple: a cyclic code $J\normaleq\F_q[H]$ is equivalent to the code $I=e\F_q[H]$ if and only if it is equivalent by a permutation from $S_e$.  This gives a simple algorithm for a code equivalence testing which is polynomial in $n$ and $q$.

It was shown in \cite{mu93} \cite{mu97} that a cyclic group of a square-free or twice square free order is a CI-group with respect to colored Cayley digraphs. In this case
Theorem~\ref{main2} implies the following
\begin{thm}\label{main3} A cyclic group of a square-free or twice square-free order is a CI-group with respect to semisimple cyclic codes.
\end{thm}
The proof of Theorem~\ref{main2} is based on the results of \cite{mu99} which
were obtained using the classification of finite simple groups (CFSG). It would be nice to find a classification-free proof of this result. Notice that if $n$ is a prime power, then the CFSG is not needed. Also for non-cyclic $p$-groups we have additional results. 

\begin{thm}\label{main1} Let $H$ be a $p$-group, $p$ a prime. Then any solving set for colored Cayley digraphs over $H$ is a solving set for semisimple group codes over $H$. In particular, if $H$ is a CI-group with respect to colored Cayley digraphs, then it is a CI-group with respect to semisimple codes.
\end{thm}
It was shown in \cite{himu} that an elementary abelian group of rank at most four is a CI-group with respect to colored digraphs. This implies the following
\begin{cor}An elementary abelian group of rank at most four is a CI-group with respect to semisimple group codes over this group.
\end{cor}

{\bf Notation.} Throughout the paper $\Omega$ denotes a finite set and $\F_q$ stands for a finite field with $q$ elements. The set of all functions from $\Omega$ to $\F_q$ is denoted as $\F_q^\Omega$. The elements of $\F_q^\Omega$ are considered either as functions or column vectors the coordinate positions of which are labelled by the elements of $\Omega$. For $f\in\F_q^\Omega$ we denote the $\omega$-th coordinate of $f$ either by $f(\omega)$ or $f_\omega$. The algebra ${\sf End}(\F_q^\Omega)$ is identified with the matrix algebra $M_\Omega(\F_q)$. The symmetric group of the set $\Omega$ is denoted by $\sym(\Omega)$. Given a permutation $g\in\sym(\Omega)$, we write $P_g$ for a permutation matrix corresponding to $g$. Notice that $P_g\in M_\Omega(\F_q)$.

\section{Preliminaries}

\subsection{Linear codes} In order to treat linear codes as combinatorial objects over finite set $\Omega$ we consider codes as linear subspaces of $\F_q^\Omega$. 
If $g\in\sym(\Omega)$, then $f^g(\omega):=f(\omega^{g^{-1}})$. Recall that two codes $\cC,\cD\leq\F_q^n$ are (permutation) equivalent if there exists $g\in\sym(\Omega)$ with $\cC^g = \cD$. An \New{automorphism} group of a code $\cC$, notation $\paut(\cC)$, consists of those $g\in\sym(\Omega)$ which satisfy $\cC^g = \cC$. A code $\cC$ is called cyclic if $\paut(\cC)$ contains a full cycle.

\subsection{Colored digraphs} Let $\Omega$ and $F$ be finite sets.  An $F$-\new{colored digraph} is a pair $\Gamma = (\Omega,c)$ where $c$ is a function $c:\Omega\times\Omega\rightarrow F$.  An adjacency matrix of $\Gamma$, $A(\Gamma)\in M_\Omega(F)$, is defined in a natural way
 $A(\Gamma)_{\omega,\omega'} = c(\omega,\omega')$. Two $F$-colored graphs $(\Omega,c)$ and $(\Omega,d)$ are isomorphic if there exists a permutation $g\in\sym(\Omega)$ such that $d(\alpha^g,\beta^g)=c(\alpha,\beta)$ for each pair $\alpha,\beta\in\Omega$. An automorphism group $\aut(\Gamma)$ consists of all isomorphisms from $\Gamma$ to itself, that is
$$g\in\aut(\Gamma)\iff \forall_{\alpha,\beta\in\Omega}\ c(\alpha,\beta) = c(\alpha^g,\beta^g).$$
If $F$ is a field, then $\aut(\Gamma)$ consists of all permutations $g\in\sym(\Omega)$ satisfying  $P_gA(\Gamma) = A(\Gamma) P_g$. Thus $\aut(\Gamma)$ coincides with the centralizer of $A(\Gamma)$ in $\sym(\Omega)$, i.e., $\aut(\Gamma) = \bC_{\sym(\Omega)}(A(\Gamma))$.

Let $H$ be a finite group and $F$ an arbitrary field. Recall that a colored Cayley digraph $\cay(H,e)$ defined by an element $e=\sum_{h\in H} e_h h \in F[H]$ has $H$ as a vertex set and an arc-coloring is defined by a function $(x,y)\mapsto e_{xy^{-1}},x,y\in H$. It's adjacency matrix will be denoted as $A_H(e)$. Clearly that $(A_H(e))_{xy} = e_{xy^{-1}}$. The set of all matrices $A_H(e),e\in F[H]$ form a subalgebra of the full matrix algebra $M_H(F)$. This subalgebra is isomorphic to the group algebra $F[H]$. Let us call matrices of the form $A_H(e)$ as $H$-\New{matrices}. Each $H$-matrix commutes with any permutation from $H_R$. Vice versa, any matrix from $M_H(F)$ which commutes with all permutations from $H_R$ is an $H$-matrix. Thus the algebra of $H$-matrices is the centralizer of $H_R$ in the full matrix algebra $M_H(F)$. 
\subsection{$2$-closed permutation groups \cite{wi2}.}
Any subgroup $G\leq \sym(\Omega)$  acts naturally on a product $\Omega\times\Omega$
as follows $(\alpha,\beta)^g:=(\alpha^g,\beta^g)$.  The orbits of this faithful
action are called $2$-\New{orbits} of $G$. The set of all $2$-orbits will be denoted as
$\Omega^2/G$. Two subgroups $G,F\leq\sym(\Omega)$ are called $2$-\New{equivalent}, notation
$G \sim_2 F$  if
$\Omega^2/G = \Omega^2/F$. The relation $\sim_2$ is an equivalence relation on the set of
all subgroups of $\sym(\Omega)$. For a given subgroup $G\leq\sym(\Omega)$ we define its $2$-closure
$G^{(2)}$ as the subgroup generated by all subgroups $2$-equivalent to $G$, that is
$$
G^{(2)}:=\sg{F\,|\, F\sim_2 G}.
$$
Notice that $G\sim_2 G^{(2)}$ and $G^{(2)} = F^{(2)}$ if and only if $G\sim_2 F$.
The operator $G\mapsto G^{(2)}$ satisfies the usual properties of a closure operator.
Notice that an intersection of two $2$-closed groups is also $2$-closed. The connection between colored digraphs and $2$-closed permutation groups is given in the statement below which is a direct consequence of Theorem 5.23 \cite{wi2} (see also Section 7.12 in \cite{KRRT})
\begin{thm}\label{2closed} An automorphism group of a colored digraph is $2$-closed. Vice versa, any $2$-closed permutation group is an automorphism group of a colored digraph.
\end{thm}
Each matrix $A\in M_\Omega(\F_q)$ is an adjacency matrix of an $\F_q$-colored digraph with vertex set $\Omega$. Therefore ${\bf C}_{\sym(\Omega)}(A)$ is a $2$-closed subgroup of $\sym(\Omega)$.

\section{Proof of main results}
Let $\cC\leq\F_q^\Omega$ be a linear code.
A \New{projector} onto $\cC$ is an endomorphism $E\in M_\Omega(\F_q)$
 such that $E^2 =E$ and ${\sf Im}(E)=\cC$. The latter condition is equivalent to
saying that the column space of $E$ coincides with $\cC$. Clearly that
${\sf Im}(E)\oplus{\sf Ker}(E)=\F_q^\Omega$. Notice that each projector onto $\cC$ is uniquely determined
by its kernel which is a subspace complementary to $\cC$. Given a subspace $\cC'$
complementary to $\cC$ in $\F_q^\Omega$,
one can define a projector $E$ onto $\cC$ by setting $Ev=v$ for $v\in\cC$ and $Ev=0$ for
$v\in\cC'$. So there is a one-to-one correspondence between projectors onto $\cC$ and
complements to $\cC$ in $\F_q^\Omega$. If a permutation matrix $P_g,g\in\sym(\Omega)$ commutes with $E$, then $\cC^g  = \cC$.
This implies the following
\begin{prop}\label{commute}
Let $E$ be a projector onto a code $\cC$. Then $\bC_{\sym(\Omega)}(E)\leq\paut(\cC)$.
\end{prop}

\begin{thm}\label{2closed}
Let $G\leq\paut(\cC)$ be a subgroup of order coprime to $q$. Then $G^{(2)}\leq \paut(\cC)$.
\end{thm}

\begin{proof}
The group algebra $\F_q[G]$ is semisimple by Maschke's Theorem. Therefore each $\F_q[G]$-module
is semisimple too. This implies that each $G$-invariant subspace of $\F_q^\Omega$ has a $G$-invariant
complement. Therefore there exists a $G$-invariant complement $\cC'$ to $\cC$ in $\F_q^\Omega$.
Let $E$ denote a projector on $\cC$ with a kernel $\cC'$. Then $E$ commutes with each $P_g, g\in G$,
or, equivalently,  $G\leq \bC_{\sym(\Omega)}(E)\leq\paut(\cC)$. Since $\bC_{\sym(\Omega)}(E)$
is $2$-closed, $G^{(2)}\leq \bC_{\sym(\Omega)}(E)\leq\paut(\cC)$.\QED
\end{proof}
By Exercise 5.28 \cite{wi2} a $2$-closure of a $p$-group is a $p$-group. This gives us the following
\begin{cor} Each Sylow $r$-subgroup of $\paut(\cC)$, $r\neq{\sf char}(\F_q)$ is $2$-closed.
\end{cor}
\subsection{Fusion control}
Let $X\leq Y\leq Z\leq \sym(\Omega)$ be arbitrary subgroups. Following \cite{laue} we say that $Y$ \new{controls fusion of $X$ in $Z$} if for any $g\in\sym(\Omega)$ the following implication holds
$$
X^g\leq Z\implies\exists_{z\in Z}\ X^{gz}\leq Y.
$$
In this case we write $Y\prec_X Z$. If $X$ is a regular subgroup of $\sym(\Omega)$, then the inequality $Y\prec_X Z$ means that for any regular subgroup $X'\leq Z$ isomorphic to $X$ there exists $z\in Z$ such that $X'^z\leq Y$.

The following properties of the relation $\prec_X$ are straightforward:
\begin{enumerate}
\item[(a)] $\prec_X$ is a transitive relation on a set of all overgroup of $X$ in $\sym(\Omega)$;

\item[(b)] if $Y\prec_X Z$ and $Y\leq W\leq Z$, then $W\prec_X Z$.
\end{enumerate}

The statement below is a direct generalization of Lemma 3.1 from \cite{babai}.
\begin{thm}\label{ssets} Let $K,L$ be two Cayley objects over $H$.
If $\aut(K)\prec_{H_R}\aut(L)$, then each solving set for $K$ is a solving set for $L$. In particular, if $K$ is  a CI-object over $H$, then so does $L$.
\end{thm}
\begin{proof} Let $S$ be a solving set for $K$. Pick an arbitrary Cayley object over $H$, say $L'$ isomorphic to $L$. Then $L'=L^g$ for some $g\in\sym(\Omega)$ and, consequently, $\aut(L')=\aut(L)^g$.
Therefore $H_R\leq \aut(L)^g$ implying $H_R^{g^{-1}}\leq \aut(L)$.
By the assumption there exists $z\in\aut(L)$ such that $H_R^{g^{-1}z^{-1}}\leq \aut(K)$. This implies that $H_R\leq \aut(K)^{zg} = \aut(K^{zg})$. Thus $K^{zg}$ is a Cayley object over $H$ isomorphic to $K$. Therefore $K^{zg} = K^s$ for some $s\in S$. Since
$zgs^{-1}\in\aut(K)\leq \aut(L)$, we conclude that $L^{zgs^{-1}}=L$, or, equivalenly, $L^{zg} = L^s$. Together with $z\in\aut(L)$ and $L^g=L'$ we obtain $L' = L^s$.\QED
\end{proof}

{\sc Proof of Theorem~\ref{main1}.} Let $P$ be a Sylow $p$-subgroup of $\paut(I)$ containing $H_R$. By Sylow's theorems $P\prec_{H_R}\paut(I)$.

Since $\F_q[P]$ is semisimple, there exists an $\F_q[P]$-invariant complement to $I$ in $\F_q[H]$, say $J$. Let $E$ be a projection on $I$ parallel to $J$. Since $E$ commutes with all permutations from $P$, it also commutes with $H_R$. Therefore $E$ is an $H$-matrix, that is $E =A_H(e)$ for some $e\in\F_q[H]$. It follows from ${\sf Im}(E) = I$ that $e\F_q[H] = I$. An equality $E^2 = E$ implies that $e$ is an idempotent. Since $P$ centralizes $A_H(e)$ and
$\bC_{\sym(H)}(A_H(e))\leq\paut(I)$, we obtain $P\leq  \bC_{\sym(H)}(A_H(e))\leq\paut(I)$. Therefore $\bC_{\sym(H)}(A_H(e))\prec_{H_R}\paut(I)$. 
Since $A_H(e)$ is the adjacency matrix of a colored Cayley graph $\cay(H,e)$, we conclude that $\bC_{\sym(H)}(A_H(e)) = \aut(\cay(H,e))$.  By Theorem~\ref{ssets} any solving set for a colored Cayley graph $\cay(H,e)$ is a solving set for a code $I$.\QED

 Notice that if $H$ is commutative, then an idempotent $e$ is unique. In the case of non-commutative $H$ a right ideal of $\F_q[H]$ may have more than one generating idempotent.

{\sc Proof of Theorem~\ref{main2}.}
By Theorem 1.8 \cite{mu99} there exists a solvable group $F$, $H_R\leq F\leq \paut(I)$ which
controls fusion of $H_R$ in $\paut(I)$. Let $\pi$ be the set of all prime divisors of $n$.
It follows from Hall's theorems that every Hall $\pi$-subgroup $F_\pi\leq F$ which contains $H_R$
controls fusion of $H_R$ in $F$. By transitivity of $\prec_{H_R}$ the group $F_\pi$ controls
fusion of $H_R$ in $\paut(I)$. Since ${\sf char}(\F_q)$ is coprime to $|F_\pi|$, there exists
a $F_\pi$-invariant complement $J$ to $I$ in $\F_q[H]$. Let $E$ denote a projector onto $I$ parallel
to $J$. Then $F_\pi\leq \bC_{\sym(H)}(E)\leq \paut(I)$ implying $\bC_{\sym(H)}(E)\prec_{H_R} \paut(I)$. Since
$H_R\leq \bC_{\sym(H)}(E)$, the matrix $E$ is circulant, that is $E=A_H(e)$ for some $e\in\F_q[H]$. It follows from $E^2=E$ and ${\sf Im}(E)=I$
that $e$ is an idempotent generating $I$. Thus $\bC_{\sym(H)}(A_H(e)) = \aut(\cay(H,e))$ controls fusion of $H_R$ in $\paut(I)$. By Theorem~\ref{ssets} any solving set for $\cay(H,e)$ is a solving set for $I$.\QED

\section{Acknowledgements}
The author is very grateful to I. Ponomarenko and M. Klin for fruitful discussions and valuable remarks.
\bibstyle{plain}


\begin{thebibliography}{30}
\comment{
\bibitem{ad} A. {\' A}d{\' a}m,
\textit{Research problem 2-10},
J.\ Combin.\ Theory \textbf{2}(1967),  393.

\bibitem{ap} B. Alspach, T.D. Parsons,
\textit{Isomorphism of circulant graphs and digraphs},
Discrete Math.\ \textbf{25}(1979),  97-108.
}
\bibitem{babai} L. Babai,
\textit{Isomorphism problem for a class of point-symmetric structures},
Acta Math.\ Acad.\ Sci.\ Hungar. \textbf{29}(1977),  329-336.
\comment{
\bibitem{babai-frankl1} L. Babai, P. Frankl,
\textit{Isomorphism of Cayley graphs I}, in:
Colloq.\ Math.\ Soc.\ J.\ Bolyai \textbf{18}.
Combinatorics, Keszthely, 1976,
North - Holland, Amsterdam, 1978,  35-52.

\bibitem{babai-frankl2}  L. Babai, P. Frankl,
\textit{Isomorphism of Cayley graphs II},
Acta Math.\ Acad.\ Sci.\ Hungar. \textbf{34}(1979),  177-183.
}
\bibitem{brand}N. Brand,\textit{ Isomorphisms of cyclic combinatorial objects}, Discrete Math. 78 (1989), no. 1-2, 73–81.

\bibitem{dobson} E. Dobson, \textit{On the Cayley isomorphism problem},
Discrete Mathematics 247 (2002), 107–116.

\comment{
\bibitem{et} B. Elspas, J. Turner,
\textit{Graphs with circulant adjacency matrices},
J.\ Combin.\ Theory \textbf{9}(1970),  297-307.
}

\bibitem{evdopon} S.Evdokimov and I.Ponomarenko.
\textit{Circulant graphs: recognizing and isomorphism testing in polynomial time},
St. Petersburg Math. J., Vol. 15 (2004), No. 6, 813–835.
\comment{
\bibitem{evdopon2} S.Evdokimov and I.Ponomarenko.
\textit{On a family of Schur rings over a finite cyclic group.}
Preprint 10/2000, PDMI, 2000.

\bibitem{fik} I.A. Farad{\u z}ev, A.A. Ivanov, M.H. Klin.
\textit{Galois correspondence between permutation groups ang cellular rings
(association schemes)},
Graphs and Combinatorics \textbf{6}(1990),  303-332.

\bibitem{fkm} I.A. Farad{\u z}ev, M.H. Klin, M.E.Muzichuk,
\textit{Cellular rings and groups of automorphisms of graphs},
In:
Investigations on Algebraic Theory of Combinatorial Objects.
Mathematics and Its Applications (Soviet Series),
I.A. Farad{\u z}ev, A.A.Ivanov, M.H.Klin, A.J.Woldar
(eds.),Kluwer Acad.\ Publ.\ \textbf{84}(1994), 1-152.


\bibitem{folklore} Folklore.

\bibitem{godsil} C.D. Godsil,
\textit{On Cayley graphs isomorphism},
Ars Combinatoria \textbf{15}(1983),  231-246.

\bibitem{go-kl-na} Ja.Ja. Gol'fand, M.H. Klin, N.L. Naimark,
\textit{The structure of S-rings over $\Z_{2^m}$}, In:``Sixteenth
All-Union Algebraic Conference", part 2, Leningrad, 1981, 195-196.

\bibitem{go-kl-na2}Ja.Ju. Gol'fand, N.L. Naimark, R. P{\" o}schel,
\textit{Schur rings over $\Z_{2^m}$},
Preprint P-MATH-01/85 Akad.\ der Wiss.\ der DDR, ZIMM, Berlin, 1985.
}
%

\bibitem{himu} M. Hirasaka, M. Muzychuk,
\textit{The elementary abelain  group of rank $4$ is a CI-group},
J.\ of Combin.\ Theory (A), \textbf{94}, 339-362 (2001) .

\bibitem{HJP} W.G. Huffman, A. Job, V. Pless,
\textit{Multipliers and generalized multipliers of cyclic
objects and cyclic codes},
J.\ of Comb.\ Th.\ (A) \textbf{62}(1993), no.\ 2,  183-215.

\bibitem{huffman} W.C. Huffman,
\textit{The equivalence of two cyclic objects on
$pq$ elements},
Discrete Math. \textbf{154}(1996), 103-127.
\comment{
\bibitem{FaKl} I.A.Farad{\u z}ev and M.H.Klin.
Computer package for calculations with coherent configurations.
Proc of the 1991 Intern Symp on Symb and Algebr. Computations,
ISSAC'91 (Bonn, July 15-17, 1991), ACM press, 219-223.
}


\comment{
\bibitem{kpl} M. Klin, V. Liskovets, R. P{\" o}schel,
\textit{Analytical enumeration of circulant graphs with prime-squared number
of vertices},
Semin.\ Lothar.\ Comb.\ \textbf{36}(1996), B36d.


\bibitem{klin-muz-pos} M.Klin, M.Muzychuk and R.P{\" o}schel
\textit{The isomorphism problem for circulant graphs via Schur rings
theory}, In the book: Codes and Association Schemes, A.Barg and S.Litsyn eds.,
DIMACS series in Discrete Mathematics and Theoretical Computer Science,
v. 56 (1999), 241-265.

}

\comment{
\bibitem{odessa75}M.H. Klin, P{\" o}schel,
\textit{The isomorphism problem for cyclic graphs with $p^2$ or $pq$
vertices},
Abstract presented at the A.A.Zykov's Seminar on Graph Theory, Odessa, 1975.
(Russian)


\bibitem{Kli-P 78} M.Ch. Klin, R. P{\" o}schel,
\textit{The K{\" o}nig problem, the isomorphism problem for cyclic graphs
and the characterization of Schur rings},
Preprint AdWd DDR, ZIMM, Berlin, March 1978. 


\bibitem{klp} M.Ch. Klin, R. P{\" o}schel,
\textit{The isomorphism problem for circulant digraphs with $p^n$
vertices}, Preprint P-34/80 Akad.\ der Wiss.\ der DDR, ZIMM, Berlin, 1980.
}
\bibitem{klp1} M.H. Klin, R. P{\" o}schel,
\textit{The K{\" o}nig problem, the isomorphism problem
for cyclic graphs and
the method of Schur rings},
Colloq.\ Math.\ Soc.\ J.\ Bolyai \textbf{25}, Algebraic methods
in graph theory,
Szeged, 1978. North-Holland, Amsterdam, 1981;  405-434.
\comment{
\bibitem{kp01} M.Klin and R.P{\" o}schel,
\textit{Circulant graphs via Schur ring theory, II.
   Automorphism groups of circulant graphs on $p^m$ vertices,
$p$ an odd prime}, manuscript, 2000.
}

\bibitem{KRRT} M. Klin, C. R{\" u}ckert, G, R{\" u}ckert, G. Tinhofer.
Algebraic Combinatorics in Mathematical Chemistry. Methods and Algorithms. I. Permutation Groups and Coherent (Cellular) Algebras. Tech. Univ. M{\" u}nchen, Fak. f. Math. Report, TUM M9510(95).

\comment{
\bibitem{koch} R. Kochend{\" o}rfer,
\textit{Lehrbuch der Gruppentheorie unter besonder Ber\"ucksichtigung der end\-lichen Gruppen},
Akad.\ Verlagsgesellschaft, Leipzig, 1966.
}

\bibitem{laue}R. Laue.
\textit{Construction of combinatorial objects, a tutorial.}
Bayreuther Mathematische Schriften, 43: 53--96, 1993.
\comment{
\bibitem{LeM} K.H. Leung, S.L. Ma, \textit{The structure of Schur rings ovter cyclic groups},
J.\ Pure Appl.\ Algebra \textbf{66}(1990), 287-302.

\bibitem{leungma}K.H. Leung, S.H. Man,
\textit{On Schur rings over cyclic groups},
Israel J.\ Math. \textbf{106}(1998), 251-267.

\bibitem{leungma2}K.H. Leung, S.H. Man, \textit{On Schur rings over cyclic
groups, II},
J.of Algebra \textbf{183}(1996), no.\ 2, 273-285.

\bibitem{li}C.H. Li,
\textit{Finite CI-groups are solvable},
Bull. London Math. Soc. 31(1999),  419-423.
}
\bibitem{li2}C.H. Li,
\textit{On isomorphisms of finite Cayley graphs - a survey},
Disc. Math. 256(2002),  301-334.
\comment{

\bibitem{lipo} V. Liskovets, R. P{\" o}schel,
\textit{On the enumeration of circulant graphs of prime-power and
square-free orders.}
Preprint Math-AL-8-1996, Technische Universit{\" a}t Dresden, 1996.

\bibitem{lipo2} V. Liskovets, R. P{\" o}schel,
\textit{Counting circulant graphs of prime-power order by
decomposing into orbit enumeration subproblems.}
Discrete Mathematics, 214 (2000), 173-191.

}




\comment{
\bibitem{mu1} M. Muzychuk,
\textit{On the structure of basic sets of Schur rings over cyclic
groups},
 J.\ of Algebra \textbf{169}(1994), no.\ 2, 655-678.

\bibitem{mu4}M. Muzychuk,
\textit{The structure of Schur rings over cyclic groups of square-free order},
 Acta Appl.\ Math.\ \textbf{52}(1998), 163-181.
}

\bibitem{mu93} M. Muzychuk,
\textit{{\' A}d{\' a}m's conjecture is true in the square-free case},
 J.\ of Combin.\ Th.\ (A) \textbf{72}(1995), no.\ 1, 118-134.

\bibitem{mu97} M. Muzychuk,
\textit{On {\' A}d{\' a}m's conjecture for circulant graphs},
 Discrete Math.\  \textbf{176}(1997),  285-298.
\comment{
\bibitem{mypreprint} M. Muzychuk,
\textit{On the number of cyclic combinatorial
objects isomorphic to  a given one},
Preprint Math-AL-16-1996, Technische Universit{\" a}t
Dresden 1996.

\bibitem{mupo} M. Muzychuk, R. P{\" o}schel,
\textit{Isomorphism criterion for circulant graphs},
 Preprint Math-AL-9-1999, Technische Universit{\" a}t Dresden, 1999.

\bibitem{muponew}  M. Muzychuk, R. P{\" o}schel,
\textit{Schur rings and circulant graphs},
manuscript.
}
\bibitem{mu99} M. Muzychuk,\textit{On the isomorphism problem
for cyclic combinatorial objects},
 Discrete Math.\ \textbf{197/198}(1999),  589-606.

\bibitem{mu04} M. Muzychuk,\textit{ A solution of isomorphism problem for circulant graphs},
Proc. of the London Math. Soc. 3 (88), 2004.



\bibitem{palfy} P.P. P{\' a}lfy,
\textit{Isomorphism problem for relational structures with a cyclic automorphism},
 Europ.\ J.\ Comb. \textbf{8}(1987),  35-43.

\bibitem{phelps} K.T.Phelps. Isomorphism problem for cyclic block
designs. Ann. Disc. Math., 34 (1987), pp. 385-392.
\comment{
\bibitem{poeschel} R. P{\" o}schel,
\textit{Untersuchungen von S-Ringen, insbensondere im Gruppenring von p-Gruppen},
 Math.\ Nachr. \textbf{60}(1974),  1-27.
}
\comment{
\bibitem{shiu} Wai-Chee Shiu, \textit{Algebraic structure of Schur rings},
Chinese Journal of Math. \textbf{21}(1993), no.\ 1, 55-71.

\bibitem{schur} I. Schur, \textit{Zur Theorie der einfach transitiven Permutationgruppen},
 S.-B. Preuss.\ Akad.\ Wiss.\ phys.-math.\ Kl.\ \textbf{18/20}(1933),  598-623.

\bibitem{stanley} R. Stanley. Enumerative Combinatorics.
Cambridge University Press, 1996, 306 p.

\bibitem{tamas} O. Tamaschke, \textit{A generalization of conjugacy
in groups},
 Rendiconti del Seminario Mathematico
dell'Universit{\'a} di
Padova \textbf{40}(1968), 408-427.

\bibitem{Wei 76} B. Weisfeiler (editor),
 \textit{On the construction and identification of graphs},
Lect.\ Notes in Math. \textbf{558}(1976).
}
\bibitem{wi} H. Wielandt,
\textit{Finite Permutation Groups},
 Academic Press, 1964, Berlin.

\bibitem{wi2} H. Wielandt,
\textit{Permutation groups through invariant relations and invariant functions},
 Lect.\ Notes., Dept.\ Math.m Ohio St.\ Univ, Colimbus, 1969.


\end{thebibliography}
\end{document}